\documentclass[11pt,letterpaper]{amsart}

\usepackage[T1]{fontenc}
\usepackage{ae}
\usepackage{amsmath,amssymb,amsthm,parskip}
\usepackage[colorlinks=true, citecolor=red, linkcolor=red]{hyperref}

\newtheorem{theorem}{Theorem}[section]
\newtheorem{proposition}[theorem]{Proposition}
\newtheorem{lemma}[theorem]{Lemma}
\newtheorem{corollary}[theorem]{Corollary}

\numberwithin{equation}{section}

\newcommand{\supp}{\operatorname{supp}}

\title[Superpolynomial Support in Irreducible Factors]{Superpolynomially Large Support in Every Irreducible Factor of Lacunary Polynomials}
\author{Bhawesh Mishra}
\email{bhaweshmishra2024@gmail.com}
\address{384 Dunn Hall, Department of Mathematical Sciences, The University of Memphis, Memphis, TN 38152}

\hypersetup{
  pdftitle={Support of Irreducible Factors of Lacunary Polynomials},
  pdfauthor={Bhawesh Mishra},
  pdfsubject={Sparse irreducible factors of lacunary polynomials},
  pdfkeywords={lacunary polynomial, sparse factorization, cyclotomic polynomial,
    Mahler measure, Hilbert irreducibility}
}

\begin{document}

\begin{abstract}
We show that there exist infinitely many polynomials $F\in\mathbb{Q}[x]$ with exactly $m$ nonzero coefficients such that every irreducible factor of \(F\) over \(\mathbb{Q}\) has
\(\exp\) \(\!\bigl(\Omega(\sqrt{m/\log m})\bigr)\) nonzero coefficients. This places multiplication in a sharply different category from powers, composition, and other algebraic operations for which reverse-sparsity principles are known. It also gives an unconditional, degree-free obstruction to sparse factorization, complementing positive factor-sparsity results that impose hypotheses on coefficient height, exponent positions, reciprocal structure, or degree bounds.
\end{abstract}

\maketitle

\section{Introduction}\label{sec:introduction}
A polynomial is called \emph{sparse}, or \emph{lacunary}, when it has few
nonzero coefficients relative to its degree. Factorization can reverse this
sparsity: cancellations may leave a polynomial with few nonzero terms even
though all of its irreducible factors have many nonzero terms. We are interested in the basic
degree-free question of whether the term count of a polynomial nevertheless
forces at least one irreducible factor to remain quantitatively sparse. We
show that it does not: the term count of the sparsest irreducible factor can
be forced to grow faster than every fixed power of the term count of the
original polynomial.

For a nonzero polynomial \(P(x)=\sum_{j=0}^{d}a_jx^j\in\mathbb{Q}[x]\),
write \(\supp(P)=\{j\in\{0,\ldots,d\}:a_j\neq0\}\); thus
\(|\supp(P)|\) is the number of nonzero terms of \(P\). For a nonconstant
\(F\in\mathbb{Q}[x]\), set
\(\lambda(F)=\min\{|\supp(G)|:G\mid F,\ G\text{ irreducible in }
\mathbb{Q}[x]\}\). The quantity \(\lambda(F)\) measures the sparsity of
the sparsest irreducible factor of \(F\). We determine how large it can be
in terms of \(|\supp(F)|\) alone, with neither degree nor coefficient height
bounded.

\begin{theorem}\label{thm:main-introduction}
There exist absolute constants \(c>0\) and \(M\) such that, for every
integer \(m\geq M\), there is a polynomial \(F\in\mathbb{Q}[x]\) with
\(|\supp(F)|=m\) such that every irreducible factor \(G\)
of \(F\) over \(\mathbb{Q}\) satisfies
\begin{equation}\label{eq:main-bound}
    |\supp(G)|
      >\exp\!\left(c\sqrt{\frac{m}{\log m}}\right).
\end{equation}
\end{theorem}

The lower bound in \eqref{eq:main-bound} dominates \(Cm^A\) for every fixed
\(C,A>0\). Consequently, for every \(A>0\) and all sufficiently large
\(m\), there is an \(m\)-term polynomial \(F\) for which
\(\lambda(F)>m^A\). In particular, no estimate
\(\lambda(F)\leq C|\supp(F)|^A\), with fixed \(C\) and \(A\), can hold
for all rational polynomials. 

\subsection{Previous Work on Sparse Factorization}\label{subsec:previous-works}
Reverse-sparsity phenomena are known for several algebraic operations.
Erd\H{o}s and R\'{e}nyi asked whether
\(|\supp(P)|\to\infty\) forces \(|\supp(P^2)|\to\infty\)
\cite{Renyi1947,Erdos}, and Schinzel proved the corresponding statement
for every fixed power \cite{Schinzel87}. Analogous rigidity results in the literature concern
polynomial composition \cite{Zannier}, rational-function composition
\cite{FuchsZannier2012}, and polynomial solutions of monic algebraic
equations with fewnomial data \cite{FuchsMantovaZannier2018}. For fixed powers, sparsity of \(P^k\) constrains \(P\); for polynomial
composition, sparsity of the composite constrains the inner polynomial; for
rational-function composition, sparsity of the composite bounds the degree of
the outer function apart from explicitly described exceptional inner functions; and for monic algebraic equations, fewnomial data impose rigidity on polynomial solutions. 
Multiplication differs because its factors vary independently and their
coefficient contributions may cancel; Theorem~\ref{thm:main-introduction}
shows that this destroys every polynomial bound relating the product's term
count to that of its sparsest irreducible factor.

Positive lacunary-factorization results retain information beyond the term
count. Schinzel developed large-gap methods systematically
\cite{Schinzel1969,Schinzel1970}, and Sawin, Shusterman, and Stoll obtained
effective irreducibility results for \(x^Nc(x^{-1})+d(x)\) under explicit
hypotheses \cite{SawinShustermanStoll2020}. Pinner and Vaaler bounded the
number of irreducible cyclotomic factors using the degree and number of
monomials, independently of the coefficients \cite{PinnerVaaler1996}.
For fixed term count and coefficient height, Filaseta, Granville, and
Schinzel obtained a bounded-term nontrivial factor for each reducible
nonreciprocal integer polynomial, and corresponding bounds for every factor
when reciprocal factors are excluded \cite{FilasetaGranvilleSchinzel2008}.
Amoroso and Sombra proved a finiteness theorem for irreducible factorizations
in bivariate Laurent-polynomial families with fixed coefficients and varying
exponents \cite{AmorosoSombra2019}. These results control factor occurrence,
number, or support only through additional structure absent from
Theorem~\ref{thm:main-introduction}.

The same distinction appears in algorithmic work too. Sparse multivariate
factorization and the complementary sparse-multiple problem are treated,
respectively, in \cite{vonZurGathenKaltofen1985} and
\cite{GiesbrechtRocheTilak2012}; Plaisted's NP-hardness results give a
separate complexity-theoretic obstruction \cite{Plaisted1977}.
Bounded-degree factors of lacunary polynomials are algorithmically accessible
in the univariate and multivariate settings
\cite{Lenstra1999Factorization,Lenstra1999SmallDegree,
KaltofenKoiran2005,AvendanoKrickSombra2007,Grenet2016}, while multilinear,
multiquadratic, and bounded-individual-degree regimes admit further
structural and algorithmic results
\cite{ChattopadhyayGrenetKoiranPortierStrozecki2021,
Volkovich2015,Volkovich2017,BhargavaSarafVolkovich2020,
ChuyoonShpilka2026}. Here degree dependence is essential: since
\(|\supp(G)|\leq\deg G+1\), every factor supplied by
Theorem~\ref{thm:main-introduction} has degree eventually exceeding every
fixed polynomial in \(m\). Exact-division algorithms likewise retain the
quotient support or total sparse output size as a complexity parameter
\cite{GiorgiGrenetPerretDuCray2021,
GiorgiGrenetPerretDuCrayRoche2022,NahshonShpilka2026}; our construction
shows that this parameter cannot generally be replaced by any fixed power
of the dividend's term count. The theorem is an algebraic output-size lower
bound, not by itself a Turing-model lower bound, since sparse bit length
also records exponents and coefficient heights.

Finally, Schinzel recorded the extremal function
\[
 K(m)=\sup\{\lambda(F):F\in\mathbb Q[x],\ |\supp(F)|=m\},
\]
where the supremum is allowed to be infinite \cite{Schinzel95}. Using a
sparse-square theorem of Verdenius, Choudhry and Schinzel proved
\(K(m)>\max\{2m,0.014m^{1.22}\}\) for \(m>2\)
\cite{CS92,Verdenius}; Bremner improved the numerical lower bounds over a
large finite range \cite{Bremner}. Theorem~\ref{thm:main-introduction}
gives \(K(m)>\exp(c\sqrt{m/\log m})\) for all sufficiently large \(m\), and
hence \(K(m)>m^A\) eventually for every fixed \(A>0\).

\subsection{A Naive Coefficient-Variety Heuristic}
One way to see the utility of our result is in the light of a \textit{naive} coefficient-variety heuristic. More specifically, let \(A,B\subseteq\mathbb Z_{\geq0}\), with \(|A|=r+1\) and
\(|B|=s+1\), and write
\[
 G_{\mathbf a}=\sum_{u\in A}a_ux^u,\qquad
 H_{\mathbf b}=\sum_{v\in B}b_vx^v,\qquad
 q_n(\mathbf a,\mathbf b)
 =\sum_{\substack{u\in A,\ v\in B\\u+v=n}}a_ub_v.
\]
Let \(T_{A,B}\subseteq\mathbb P^r\times\mathbb P^s\) be the locus where
all \(a_u,b_v\) are nonzero. For \(S\subseteq A+B\), define
\[
 X_{A,B,S}=
 \left(T_{A,B}\cap\bigcap_{n\notin S}V(q_n)\right)
 \setminus\bigcup_{n\in S}V(q_n),
\]
where \(n\) ranges over \(A+B\). Then \(X_{A,B,S}(\mathbb Q)\)
parametrizes, up to independent rescaling, factor pairs with supports
\(A,B\) and product support \(S\). Thus the direct approach is a
rational-point search on these varieties.

Each \(q_n\) has bidegree \((1,1)\). If \(r+s-1\) required divisors meet
transversely, their intersection is a smooth curve \(C\). Adjunction
\cite[Chapter~II, Proposition~8.20, p.~182]{Hartshorne1977} and
Stirling's formula give \(g(C)=\exp(\Theta_c(r+s))\) as
\(r,s\to\infty\) with
\(c\le r/s\le c^{-1}\). Because the \(q_n\) are special, this is only a
complete-intersection heuristic.

No unconditional general algorithm is known for deciding whether a
smooth projective curve of genus at least two over \(\mathbb Q\) has a
rational point
\cite[Questions~1.1--1.2 and the subsequent discussion,
pp.~181--182]{BruinStoll2008}; known proofs of Faltings's theorem provide
no such algorithm \cite[\S1, p.~484]{Stoll2026}. Moreover, the incidence
equations encode cancellations, not irreducibility. Indeed, \(x^2-t\)
is reducible over \(\mathbb Q\) precisely when \(t\) is a rational
square, while both squares and nonsquares are Zariski dense in
\(\mathbb A^1_{\mathbb Q}\); irreducibility therefore cannot generally
be imposed by deleting a Zariski-closed subset. Absent additional
structure, the naive strategy is decisively out of reach as a uniform
proof method: it requires solving a rational-point problem for which no
unconditional general procedure is known, and even success there leaves
the independent arithmetic problem of irreducibility. This difficulty also sheds light on why analogous positive sparsity results in literature are obtained by restricting either the height of the coefficients or degrees of factors or working within specialized subfamilies of polynomials (see Subsection \ref{subsec:previous-works}). 

\subsection{An Overview of Our Proof}\label{subsec:proof-architecture}
We establish Theorem \ref{thm:main-introduction} by separating the construction into three distinct
tasks: forcing large support in a common cyclotomic factor, forcing large
support in every possible cofactor used in the construction, and combining
those cofactors into a single irreducible polynomial without cancellation.

Let $N$ be square-free with $r$ prime divisors, all congruent to $2$ or $3$
modulo $5$.  A formula for cyclotomic values at roots of unity due to
Bzd\k{e}ga, Herrera-Poyatos, and Moree
\cite{BzdegaHerreraMoree2018} gives two primitive fifth roots
$\zeta_+$ and $\zeta_-$ for which
\[
    |\Phi_N(\zeta_+)|=\alpha^{2^{r-1}}
    \qquad\text{and}\qquad
    |\Phi_N(\zeta_-)|=\alpha^{-2^{r-1}},
\]
where $\alpha>1$ is an absolute constant.  For each prime $p\mid N$, put
\[
    C_p(x)=\frac{x^N-1}{x^{N/p}-1}
          =1+x^{N/p}+\cdots+x^{(p-1)N/p},
    \qquad
    D_p(x)=\frac{C_p(x)}{\Phi_N(x)}.
\]
The polynomial $C_p$ has exactly $p$ terms, whereas $\Phi_N$ and every $D_p$
are monic products of cyclotomic polynomials and hence have Mahler measure
one.  Moreover, $|C_p(\zeta_-)|$ is bounded below uniformly in $p$ and $N$.
It follows that $D_p$ is exponentially large at $\zeta_-$, uniformly for
every $p\mid N$.

An inequality of Akhtari and Vaaler implies that if a polynomial $P$ has
Mahler measure one and $s$ nonzero coefficients, then
\[
    |P(z)|\leq 2^{s-1}\qquad (|z|=1)
\]
\cite{AkhtariVaaler2019}.  Applying this estimate at $\zeta_+$ to $\Phi_N$
and at $\zeta_-$ to every $D_p$ yields an absolute constant $c_0>0$ such
that
\[
    |\operatorname{supp}(\Phi_N)|>c_0 2^r,
    \qquad
    |\operatorname{supp}(D_p)|>c_0 2^r
    \quad (p\mid N).
\]
The two opposite cyclotomic evaluations are the amplification mechanism:
the large value controls the common factor, while the small value of that
factor makes all the quotients large simultaneously.

The cyclotomic factorizations of the $D_p$ show that
\(\gcd_{p\mid N}D_p=1\).  Choose a list
\(\ell_0,\ldots,\ell_s\) of prime divisors of $N$, containing every prime
divisor at least once, and choose widely separated exponents $h_j$.  The
generic linear combination
\[
    \sum_{j=0}^s T_jx^{h_j}D_{\ell_j}(x)
\]
is irreducible over $\mathbb{Q}(T_0,\ldots,T_s)$ because its coefficients
have greatest common divisor one.  Hilbert irreducibility then supplies
nonzero rational specializations $T_j=t_j$ for which
\[
    Q(x)=\sum_{j=0}^s t_jx^{h_j}D_{\ell_j}(x)
\]
is irreducible over $\mathbb{Q}$ \cite{Serre2008}.  The shifts $h_j$ are
chosen so that the supports of the summands are disjoint.  Consequently
there is no cancellation, and
\[
    F(x)=\sum_{j=0}^s t_jx^{h_j}C_{\ell_j}(x)
        =\Phi_N(x)Q(x)
\]
has exactly $\sum_j\ell_j$ terms, while the support of $Q$ is the disjoint
sum of the supports of the corresponding $D_{\ell_j}$.  Thus $\Phi_N$ and
$Q$ are precisely the two nonassociate irreducible factors of $F$, and both
have more than $c_0 2^r$ terms.

Finally, primes congruent to $2$ or $3$ modulo $5$ are chosen up to a
parameter $X$.  The prime number theorem in arithmetic progressions gives
\[
    r\asymp \frac{X}{\log X},
    \qquad
    \sum_{\substack{p\leq X\\ p\equiv 2,3\pmod 5}}p
       \asymp \frac{X^2}{\log X}
\]
\cite{Koukoulopoulos2019}.  Taking $X$ of order
\(\sqrt{m\log m}\) gives
\(r\gg\sqrt{m/\log m}\), while leaving enough room to represent every
sufficiently large $m$ as $\sum_j\ell_j$ by adding copies of the primes $2$
and $3$.  Hence the construction applies to every sufficiently large $m$,
not merely to a subsequence, and
\[
    2^r>\exp\!\left(c\sqrt{\frac{m}{\log m}}\right)
\]
after adjusting the absolute constant $c$.

\section{Mahler Measure and Exponential Amplification at Primitive Fifth Roots}
\subsection{A large value on the unit circle forces many coefficients}

The first ingredient is a way to deduce that a polynomial has many
nonzero coefficients without knowing those coefficients individually.
If
$P(x)=a\prod_{j=1}^d(x-\alpha_j)$ is a nonzero complex polynomial, its
Mahler measure is
\begin{equation}\label{eq:Mahler-measure}
 M(P)=|a|\prod_{j=1}^d\max\{1,|\alpha_j|\}.
\end{equation}
Every cyclotomic polynomial has Mahler measure one, since all of its roots
lie on the unit circle and it is monic.  The same is true of every monic
product of cyclotomic polynomials.

The following theorem controls each coefficient in terms of the number of
nonzero monomials rather than the degree.  That distinction is essential
here, because the gaps between the exponents will be very large. This is \cite[Theorem~1.1, p.~1426]{AkhtariVaaler2019}.  The source writes
the number of monomials as $N+1$; in \eqref{eq:AV-bound}, its $N$ is
$s-1$, and its exponent $m_j$ is denoted by $n_j$.

\begin{theorem}\label{thm:AV}
Let
\[
 P(x)=c_0x^{n_0}+c_1x^{n_1}+\cdots+c_{s-1}x^{n_{s-1}}
\]
be a nonzero polynomial in $\mathbb C[x]$, where
$0\le n_0<n_1<\cdots<n_{s-1}$.  Then
\begin{equation}\label{eq:AV-bound}
 |c_j|\le \binom{s-1}{j}M(P)\qquad(0\le j\le s-1).
\end{equation}
\end{theorem}

For a polynomial of Mahler measure one, Theorem~\ref{thm:AV} immediately
bounds the sum of the absolute values of its coefficients.  Evaluation at
a point on the unit circle is therefore bounded as well.

\begin{corollary}\label{cor:unit-evaluation}
Let $P\in\mathbb C[x]$ have Mahler measure one and exactly $s$ nonzero
coefficients.  If $|z|=1$, then $|P(z)|\leq 2^{s-1}$. If, in addition,
$P(z)\neq0$, then
\begin{equation}\label{eq:support-from-value}
 s\ge 1+\frac{\log|P(z)|}{\log2}.
\end{equation}
\end{corollary}

\begin{proof}
Write the nonzero terms of $P$ in increasing order of their exponents, as
in Theorem~\ref{thm:AV}.  Since $|z|=1$, the triangle inequality and
\eqref{eq:AV-bound} give
\[
 |P(z)|\le\sum_{j=0}^{s-1}|c_j|
 \le\sum_{j=0}^{s-1}\binom{s-1}{j}=2^{s-1}.
\]
If $P(z)\neq0$, taking logarithms and rearranging proves
\eqref{eq:support-from-value}.
\end{proof}

Thus a value of size $\exp(L)$ at a point of the unit circle forces at
least a constant multiple of $L$ nonzero coefficients.  The remainder of
the proof is devoted to producing such values simultaneously for the two
factors in the construction.

\subsection{Exponential amplification at primitive fifth roots}

Write $\omega(n)$ for the number of distinct prime divisors of $n$, and
$\Omega(n)$ for the number of prime divisors counted with multiplicity.
The special role of the residue classes $2$ and $3$ modulo $5$ comes from
the following cyclotomic evaluation.

\begin{theorem}[Bzd\k{e}ga--Herrera-Poyatos--Moree, specialized to
$m=5$]\label{thm:fifth-root}
Let $n>1$ be coprime to $5$, and suppose that no prime divisor of $n$ is
congruent to $1$ or $-1$ modulo $5$.  If $\xi$ is a primitive fifth root
of unity, then
\begin{equation}\label{eq:fifth-root-value}
 \log|\Phi_n(\xi)|
 =(-1)^{\Omega(n)-1}2^{\omega(n)-1}\log|1+\xi|.
\end{equation}
\end{theorem}

This is precisely the $m=5$ case of
\cite[Corollary~22, pp.~223--224]{BzdegaHerreraMoree2018}; the other
three values of $m$ in the source are not used.  Its notation $\xi_5$
has been replaced by $\xi$, and its quantity $\gamma_5$ is $1+\xi$.

We now identify two primitive fifth roots at which the right-hand side of
\eqref{eq:fifth-root-value} has the same magnitude but opposite signs.
Let $\zeta=e^{2\pi i/5}$ and put $\alpha=|1+\zeta|$.  Since
$\alpha=2\cos(\pi/5)>1$, this is a fixed real number greater than one.
Complex conjugation gives
$|1+\zeta|=|1+\zeta^4|$ and
$|1+\zeta^2|=|1+\zeta^3|$.  On the other hand,
the roots of $\Phi_5$ are $\zeta,\zeta^2,\zeta^3,\zeta^4$, and
$\Phi_5(-1)=1$; hence
\[
 \prod_{a=1}^4(1+\zeta^a)=\Phi_5(-1)=1.
\]
Taking absolute values shows that
$|1+\zeta^2|=|1+\zeta^3|=\alpha^{-1}$.

\begin{corollary}\label{cor:large-small}
Let $N$ be square-free, and suppose that all its $r$ prime divisors are
congruent to $2$ or $3$ modulo $5$.  If $r\ge2$, there are primitive
fifth roots $\zeta_+$ and $\zeta_-$ such that
\begin{equation}\label{eq:large-small}
 |\Phi_N(\zeta_+)|=\alpha^{2^{r-1}},
 \qquad
 |\Phi_N(\zeta_-)|=\alpha^{-2^{r-1}}.
\end{equation}
\end{corollary}

\begin{proof}
The hypotheses of Theorem~\ref{thm:fifth-root} are all satisfied: no
prime divisor of $N$ is $5$, and the only permitted residue classes are
$2$ and $3$, neither of which is congruent to $1$ or $-1$ modulo $5$.
Because $N$ is square-free, $\Omega(N)=\omega(N)=r$.  Consequently,
for either $\xi=\zeta$ or $\xi=\zeta^2$, formula
\eqref{eq:fifth-root-value} becomes
\[
 \log|\Phi_N(\xi)|
 =(-1)^{r-1}2^{r-1}\log|1+\xi|.
\]
The two possible values of $\log|1+\xi|$ are $\log\alpha$ and
$-\log\alpha$.  Assign the name $\zeta_+$ to the root for which the RHS has positive sign, and $\zeta_-$ to the other root.  Exponentiating
gives \eqref{eq:large-small}.
\end{proof}

The large value in \eqref{eq:large-small} will force $\Phi_N$ to have
many coefficients.  The small value will be used in the opposite way:
dividing by $\Phi_N(\zeta_-)$ will make a whole family of quotients large.

\section{Construction Using Sparse Geometric Sums}

\subsection{Sparse Geometric Sums}
Let $N$ satisfy the hypotheses of Corollary~\ref{cor:large-small}.  For
each prime $p$ dividing $N$, define
\begin{equation}\label{eq:C-and-D}
 C_p(x)=\frac{x^N-1}{x^{N/p}-1}
       =1+x^{N/p}+\cdots+x^{(p-1)N/p},
 \qquad
 D_p(x)=\frac{C_p(x)}{\Phi_N(x)}.
\end{equation}
The exponents $0,N/p,\ldots,(p-1)N/p$ are distinct and their coefficients
are all one, so $C_p$ has exactly $p$ nonzero coefficients.  We next verify
that $D_p$ is a polynomial and record its cyclotomic factors explicitly.

For every positive integer $a$,
$x^a-1=\prod_{d\mid a}\Phi_d(x)$.  Applying this identity to the numerator
and denominator in \eqref{eq:C-and-D} gives
\[
 C_p(x)
 =\frac{\prod_{d\mid N}\Phi_d(x)}
        {\prod_{d\mid N/p}\Phi_d(x)}
 =\prod_{\substack{d\mid N\\d\nmid N/p}}\Phi_d(x).
\]
Because $N$ is square-free, a divisor $d$ of $N$ fails to divide $N/p$
exactly when $p$ divides $d$.  Indeed, if $p$ does not divide $d$, then
every prime divisor of $d$ occurs in the square-free integer $N/p$, so
$d$ divides $N/p$; the converse is immediate because $p$ does not divide
$N/p$.  Therefore
\begin{equation}\label{eq:C-factorization}
 C_p(x)=\prod_{\substack{d\mid N\\p\mid d}}\Phi_d(x).
\end{equation}
The factor with $d=N$ is $\Phi_N$, so division by it gives
\begin{equation}\label{eq:D-factorization}
 D_p(x)=\prod_{\substack{d\mid N,\ p\mid d\\d\ne N}}\Phi_d(x).
\end{equation}
Every index $d$ in \eqref{eq:D-factorization} is greater than one, so each
factor $\Phi_d$ is monic with constant coefficient one.  Consequently,
$D_p\in\mathbb Z[x]$ is monic, has constant coefficient one, and has
Mahler measure one.

The family of all the quotients has no common nonconstant factor.  This
fact will later make it possible to combine them into one irreducible
polynomial.

\begin{lemma}\label{lem:gcd}
As $p$ ranges over the prime divisors of $N$, the polynomials $D_p$ have
greatest common divisor one in $\mathbb Q[x]$.
\end{lemma}

\begin{proof}
Cyclotomic polynomials are irreducible over $\mathbb Q$ and two of them
with different indices are nonassociate.  Suppose that $\Phi_d$ divided $D_p$ for every
prime $p$ dividing $N$.  Formula \eqref{eq:D-factorization} would then
imply that every prime divisor $p$ of $N$ divides $d$.  Since $N$ is
square-free, this says that $N$ divides $d$.  The same formula also
requires $d$ to divide $N$, and hence $d=N$.  But the factor $\Phi_N$ is
excluded from every product in \eqref{eq:D-factorization}.  No
irreducible polynomial can therefore divide all the $D_p$.
\end{proof}

We can now apply the two primitive fifth roots from
Corollary~\ref{cor:large-small}.  The root $\zeta_+$ controls the common
factor, whereas $\zeta_-$ controls every quotient at once.

\begin{proposition}\label{prop:simultaneous-support}
There is an absolute constant $c_0>0$ such that, whenever $r$ is
sufficiently large,
\begin{equation}\label{eq:simultaneous-support}
 |\supp(\Phi_N)|>c_0 2^r
 \quad\text{and}\quad
 |\supp(D_p)|>c_0 2^r
 \quad\text{for every prime }p\mid N.
\end{equation}
\end{proposition}

\begin{proof}
Write $s(P)=|\supp(P)|$.  Corollary~\ref{cor:unit-evaluation} says
that if $P$ has Mahler measure one, $|z|=1$, and $P(z)\ne0$, then
\[
 s(P)\ge 1+\frac{\log|P(z)|}{\log2}.
\]

First consider $\Phi_N$.  Since $\Phi_N$ is cyclotomic, it has Mahler
measure one.  The first equality in \eqref{eq:large-small} gives
\[
 |\Phi_N(\zeta_+)|=\alpha^{2^{r-1}},
\]
which is nonzero.  Applying Corollary~\ref{cor:unit-evaluation} with
$P=\Phi_N$ and $z=\zeta_+$ therefore gives
\begin{equation}\label{eq:Phi-support}
 |\supp(\Phi_N)|
 \ge 1+\frac{\log|\Phi_N(\zeta_+)|}{\log2}
 =1+\frac{2^{r-1}\log\alpha}{\log2}.
\end{equation}

We next treat the quotients.  Define
\[
 \delta_{\min}=\min_{1\le a\le4}|1-\zeta_-^a|,
 \qquad
 \delta_{\max}=\max_{1\le a\le4}|1-\zeta_-^a|,
 \qquad
 \kappa=\frac{\delta_{\min}}{\delta_{\max}}.
\]
Because $\zeta_-$ has order five, each of
$\zeta_-,\zeta_-^2,\zeta_-^3,\zeta_-^4$ is different from $1$.
Consequently, $\delta_{\min}>0$ and hence $\kappa>0$.  Moreover, the
powers $\zeta_-^a$, $1\le a\le4$, run through all the nontrivial fifth
roots of unity.  Thus $\delta_{\min}$, $\delta_{\max}$, and $\kappa$
are absolute constants, independent of the particular choice of the
primitive fifth root $\zeta_-$.

Fix a prime $p\mid N$.  Since every prime divisor of $N$ is congruent
to $2$ or $3$ modulo $5$, neither $N$ nor $N/p$ is divisible by $5$.
It follows that both $\zeta_-^N$ and $\zeta_-^{N/p}$ are nontrivial
fifth roots of unity.  In particular, the denominator below is
nonzero, and the definition of $C_p$ gives
\[
 |C_p(\zeta_-)|
 =\left|\frac{\zeta_-^N-1}{\zeta_-^{N/p}-1}\right|
 =\frac{|\zeta_-^N-1|}{|\zeta_-^{N/p}-1|}.
\]
The numerator is at least $\delta_{\min}$, whereas the denominator is
at most $\delta_{\max}$.  Therefore
\begin{equation}\label{eq:C-lower-value}
 |C_p(\zeta_-)|
 \ge\frac{\delta_{\min}}{\delta_{\max}}
 =\kappa
 \qquad(p\mid N).
\end{equation}

Since $C_p=\Phi_ND_p$, we have
\[
 |D_p(\zeta_-)|
 =\frac{|C_p(\zeta_-)|}{|\Phi_N(\zeta_-)|}.
\]
The second equality in \eqref{eq:large-small} states that
\[
 |\Phi_N(\zeta_-)|=\alpha^{-2^{r-1}}.
\]
Combining this with \eqref{eq:C-lower-value} yields
\begin{equation}\label{eq:D-large-value}
 |D_p(\zeta_-)|
 \ge\frac{\kappa}{\alpha^{-2^{r-1}}}
 =\kappa\alpha^{2^{r-1}}.
\end{equation}
In particular, $D_p(\zeta_-)\ne0$.

By \eqref{eq:D-factorization}, every $D_p$ is a monic product of
cyclotomic polynomials and hence has Mahler measure one.  Applying
Corollary~\ref{cor:unit-evaluation} with $P=D_p$ and $z=\zeta_-$,
and then using \eqref{eq:D-large-value}, gives
\begin{align}
 |\supp(D_p)|
 &\ge 1+\frac{\log|D_p(\zeta_-)|}{\log2} \notag\\
 &\ge 1+\frac{\log\kappa+2^{r-1}\log\alpha}{\log2}.
 \label{eq:D-support}
\end{align}

Set
\[
 c_0=\frac{\log\alpha}{4\log2}>0.
\]
Since $\kappa>0$ and $\alpha>1$ are absolute constants, there is an
absolute integer $r_0$ such that
\[
 -\log\kappa\le 2^{r-2}\log\alpha
 \qquad(r\ge r_0).
\]
For such $r$,
\[
 \log\kappa+2^{r-1}\log\alpha
 \ge 2^{r-2}\log\alpha.
\]
Hence \eqref{eq:D-support} implies, uniformly for every prime $p\mid N$,
\[
 |\supp(D_p)|
 \ge 1+\frac{2^{r-2}\log\alpha}{\log2}
 =1+c_0 2^r
 >c_0 2^r.
\]
Likewise, \eqref{eq:Phi-support} gives
\[
 |\supp(\Phi_N)|
 \ge 1+\frac{2^{r-1}\log\alpha}{\log2}
 >\frac{2^{r-2}\log\alpha}{\log2}
 =c_0 2^r.
\]
This proves both assertions in
\eqref{eq:simultaneous-support}.
\end{proof}

\subsection{Combining all the quotients into one irreducible factor}

We next turn the family $D_p$ into a single irreducible polynomial while
retaining every coefficient counted in Proposition~\ref{prop:simultaneous-support}.
The first lemma explains why independent coefficient indeterminates give
irreducibility before specialization.

\begin{lemma}[Generic linear combination]\label{lem:generic}
Let $A_0,\ldots,A_s\in\mathbb Q[x]$ be nonzero polynomials with
$\gcd(A_0,\ldots,A_s)=1$, and suppose that at least one $A_j$ is
nonconstant.  If $T_0,\ldots,T_s$ are algebraically independent over
$\mathbb Q$, then
\begin{equation}\label{eq:generic-polynomial}
 P(x,T_0,\ldots,T_s)=\sum_{j=0}^sT_jA_j(x)
\end{equation}
is irreducible in $\mathbb Q(T_0,\ldots,T_s)[x]$.
\end{lemma}

\begin{proof}
First regard $P$ as an element of
$R=\mathbb Q[x,T_0,\ldots,T_s]$.  If $P=UV$ in $R$, then additivity of
the total degree in $T_0,\ldots,T_s$ shows that one of $U,V$ is
independent of all these indeterminates, because $P$ has total degree one
in them.  After interchanging the two factors, write that factor as
$U=g(x)\in\mathbb Q[x]$.  Comparison of the coefficient of each $T_j$ in
$P=gV$ shows that $g$ divides every $A_j$.  Their greatest common divisor
is one, so $g$ is a nonzero rational constant.  Hence $P$ is irreducible
in $R$.

Put $K=\mathbb Q(T_0,\ldots,T_s)$. Suppose that $P=GH$ in $K[x]$. Choose
nonzero $a,b\in\mathbb Q[T_0,\ldots,T_s]$ so that $U=aG$ and $V=bH$
belong to $R$.  Clearing denominators gives $abP=UV$.  The ring $R$ is a
unique factorization domain, and its irreducible element $P$ is therefore
prime.  Moreover, $P$ cannot divide $ab$, since $ab$ has degree zero in
$x$ and $P$ has positive degree in $x$.  Thus $P$ divides $U$ or $V$;
assume $U=PW$.  Cancelling $P$ from $abP=UV$ gives $ab=WV$.  The left
side has degree zero in $x$, so both $W$ and $V$ have degree zero in
$x$.  Since $H=V/b$, the polynomial $H$ is a unit of $K[x]$.  Every
factorization of $P$ in $K[x]$ is therefore trivial.
\end{proof}

The rational specialization will use two results of Serre.  We first state
his terminology.  Let $k$ be a field of characteristic zero and let $V$ be
an irreducible variety over $k$.  A subset of $V(k)$ is of type
$\mathrm{(C_1)}$ if it is contained in $W(k)$ for some proper closed
subvariety $W$ of $V$.  It is of type $\mathrm{(C_2)}$ if it is contained
in $\pi(V'(k))$ for some irreducible variety $V'$ satisfying
$\dim V'=\dim V$ and some generically surjective morphism
$\pi\colon V'\to V$ of degree at least two.  A subset of $V(k)$ is
\emph{thin} if it is contained in a finite union of subsets of types
$\mathrm{(C_1)}$ and $\mathrm{(C_2)}$.  The variety $V$ has the
\emph{Hilbert property} over $k$ when $V(k)$ is not thin
\cite[Definitions~3.1.1--3.1.2, p.~19]{Serre2008}. The following is taken from \cite[Proposition~3.3.5, p.~24]{Serre2008}, in the polynomial formulation given immediately before that proposition.

\begin{samepage}
\begin{proposition}[Serre's specialization proposition]\label{prop:Serre}
Let $V$ be an irreducible variety over a field $k$ of characteristic
zero, and let
\[
 f(X)=X^d+a_1X^{d-1}+\cdots+a_d
\]
be irreducible over the function field $k(V)$.  Outside a thin subset
of $V(k)$, the functions $a_i$ have no pole and the specialized polynomial
is irreducible over $k$.
\end{proposition}

\end{samepage}

\begin{theorem}[Hilbert]\label{thm:Serre-Hilbert}
If $k$ is a number field, then every affine space $\mathbb A^n$ and every
projective space $\mathbb P^n$ has the Hilbert property over $k$.
\end{theorem}

This is \cite[Theorem~3.4.1, p.~25]{Serre2008}.

We require the specialization coordinates to be nonzero, because a zero
coordinate would remove coefficients whose presence is needed in the
support count.  The preceding two results give exactly that additional
freedom.

\begin{corollary}[Hilbert specialization]
\label{cor:Hilbert}
Let $P(T_0,\ldots,T_s,x)\in\mathbb Q[T_0,\ldots,T_s,x]$ have positive
degree in $x$, and suppose that $P$ is irreducible over the rational
function field $\mathbb Q(T_0,\ldots,T_s)$.  Given finitely many proper rational
hyperplanes in $\mathbb A^{s+1}$, there is a rational point
$(t_0,\ldots,t_s)$ outside their union for which
$P(t_0,\ldots,t_s,x)$ is irreducible in $\mathbb Q[x]$.
\end{corollary}

\begin{proof}
Let $L(T_0,\ldots,T_s)$ be the leading coefficient of $P$ as a
polynomial in $x$.  Over the function field of $\mathbb A^{s+1}$, the
monic polynomial $P/L$ is irreducible.  Proposition~\ref{prop:Serre}
places the rational points giving a reducible specialization in a thin
set.  The zero set of $L$ is a proper closed subvariety, as is every
prescribed hyperplane; their rational points are thin sets of the first
type.  A finite union of these sets is thin.  By
Theorem~\ref{thm:Serre-Hilbert}, the rational points of
$\mathbb A^{s+1}$ are not thin, so there is a rational point outside that
union.  At this point the leading coefficient is nonzero, the degree in
$x$ is preserved, and the specialization is irreducible over $\mathbb Q$.
\end{proof}

We now assemble the geometric sums.  Multiplication by the powers
$x^{j(N+1)}$ places the exponent set of each polynomial strictly before
the exponent set belonging to the next index.  This will preserve both
the exact number of coefficients of the resulting polynomial and the
sum of the support sizes of its quotient by $\Phi_N$.

\begin{proposition}\label{prop:assembly}
Let $N$ be square-free and have at least two prime divisors.  Let
$\ell_0,\ell_1,\ldots,\ell_s$ be prime divisors of $N$, with every prime
divisor of $N$ occurring at least once, and put $h_j=j(N+1)$.  There are
nonzero rational numbers $t_0,\ldots,t_s$ such that
\begin{equation}\label{eq:F-and-Q}
 F(x)=\sum_{j=0}^s t_jx^{h_j}C_{\ell_j}(x)
     =\Phi_N(x)Q(x),
 \qquad
 Q(x)=\sum_{j=0}^s t_jx^{h_j}D_{\ell_j}(x),
\end{equation}
where $Q$ is irreducible over $\mathbb Q$.  Moreover,
\begin{equation}\label{eq:support-counts}
 |\supp(F)|=\sum_{j=0}^s\ell_j,
 \qquad
 |\supp(Q)|=\sum_{j=0}^s|\supp(D_{\ell_j})|,
\end{equation}
and $\Phi_N,Q$ are precisely the two nonassociate irreducible factors of
$F$ over $\mathbb Q$.
\end{proposition}

\begin{proof}
We first show that the shifts $h_j=j(N+1)$ place the summands in
pairwise disjoint blocks of exponents.  For every prime $p\mid N$, the
geometric sum in \eqref{eq:C-and-D} satisfies
$\deg C_p=N-N/p<N$.  Since $C_p=\Phi_ND_p$ and $\Phi_N$ is
nonconstant, we also have $\deg D_p<\deg C_p<N$.  Consequently,
\[
 \supp\bigl(x^{h_j}C_{\ell_j}\bigr),
 \ \supp\bigl(x^{h_j}D_{\ell_j}\bigr)
 \subseteq
 \{h_j,h_j+1,\ldots,h_j+N-1\}.
\]
If $j<k$, then $h_k-h_j=(k-j)(N+1)\ge N+1$.  Thus
$h_j+N-1<h_k$, so the exponent blocks belonging to two different
indices are disjoint, both for the polynomials $x^{h_j}C_{\ell_j}$ and
for the polynomials $x^{h_j}D_{\ell_j}$.

Suppose for the moment that $t_0,\ldots,t_s$ are arbitrary nonzero
rational numbers.  Multiplication by $x^{h_j}$ translates the support
by $h_j$, while multiplication by $t_j$ does not change it.  Since the
translated supports are pairwise disjoint, no coefficient contributed
by one summand can cancel a coefficient contributed by another.
Therefore
\[
 \supp(F)
 =
 \bigsqcup_{j=0}^s
 \bigl(h_j+\supp(C_{\ell_j})\bigr),
\]
where the union is disjoint.  The polynomial $C_{\ell_j}$ is the
geometric sum
$1+x^{N/\ell_j}+\cdots+x^{(\ell_j-1)N/\ell_j}$, so it has exactly
$\ell_j$ nonzero coefficients.  It follows that
\[
 |\supp(F)|
 =
 \sum_{j=0}^s|\supp(C_{\ell_j})|
 =
 \sum_{j=0}^s\ell_j.
\]
This proves the first equality in \eqref{eq:support-counts}.

The same argument gives
\[
 \supp(Q)
 =
 \bigsqcup_{j=0}^s
 \bigl(h_j+\supp(D_{\ell_j})\bigr),
\]
and therefore
\[
 |\supp(Q)|
 =
 \sum_{j=0}^s|\supp(D_{\ell_j})|.
\]
This proves the second equality in \eqref{eq:support-counts}.

Moreover, $C_p=\Phi_ND_p$ for every prime $p\mid N$.  Hence
\begin{align*}
 F(x)
 &=
 \sum_{j=0}^s t_jx^{h_j}C_{\ell_j}(x)\\
 &=
 \sum_{j=0}^s
 t_jx^{h_j}\Phi_N(x)D_{\ell_j}(x)\\
 &=
 \Phi_N(x)
 \sum_{j=0}^s t_jx^{h_j}D_{\ell_j}(x)
 =
 \Phi_N(x)Q(x),
\end{align*}
which is the factorization asserted in \eqref{eq:F-and-Q}.

It remains to choose the $t_j$ so that $Q$ is irreducible.  Put
$A_j(x)=x^{h_j}D_{\ell_j}(x)$ for $0\le j\le s$.  We claim that
$\gcd(A_0,\ldots,A_s)=1$.

Indeed, let $g\in\mathbb Q[x]$ divide every $A_j$.  Since $h_0=0$, we
have $A_0=D_{\ell_0}$, and hence $g\mid D_{\ell_0}$.  The polynomial
$D_{\ell_0}$ has constant coefficient one, so it is not divisible by
$x$.  Therefore $x\nmid g$, and consequently
$\gcd(g,x^{h_j})=1$ for every $0\le j\le s$.

For every $j$, we know that $g\mid x^{h_j}D_{\ell_j}$.  Since $g$ is
coprime to $x^{h_j}$, Euclid's lemma gives $g\mid D_{\ell_j}$.  Every
prime divisor $p$ of $N$ occurs among the $\ell_j$.  Consequently,
$g$ divides $D_p$ for every prime $p\mid N$.  Lemma~\ref{lem:gcd}
states that these polynomials have greatest common divisor one, so
$g$ must be a nonzero rational constant.  This proves
\[
 \gcd\bigl(x^{h_0}D_{\ell_0},\ldots,
           x^{h_s}D_{\ell_s}\bigr)=1.
\]

Because $N$ has at least two prime divisors and every prime divisor
occurs in the list, the list contains at least two entries; hence
$s\ge1$.  In particular,
$A_1(x)=x^{N+1}D_{\ell_1}(x)$ is nonconstant.
Lemma~\ref{lem:generic} therefore applies to $A_0,\ldots,A_s$ and
shows that the generic polynomial
\[
 \mathcal Q(T_0,\ldots,T_s,x)
 =
 \sum_{j=0}^sT_jx^{h_j}D_{\ell_j}(x)
\]
is irreducible over $\mathbb Q(T_0,\ldots,T_s)$.

Apply Corollary~\ref{cor:Hilbert} to $\mathcal Q$, excluding the
coordinate hyperplanes $T_0=0,T_1=0,\ldots,T_s=0$.  The corollary
supplies a rational point
$(t_0,\ldots,t_s)\in\mathbb Q^{s+1}$ outside all these hyperplanes for
which $Q(x)=\mathcal Q(t_0,\ldots,t_s,x)$ is irreducible over
$\mathbb Q$.  Being outside the coordinate hyperplanes means precisely
that every $t_j$ is nonzero.  Thus the disjoint-support arguments
proving the two equalities in \eqref{eq:support-counts} remain valid
for this specialization.

Finally, we show that $Q$ and $\Phi_N$ are not associates.  For every
$j<s$,
\[
 \deg\bigl(x^{h_j}D_{\ell_j}\bigr)
 \le h_j+N-1
 \le h_{s-1}+N-1
 =h_s-2.
\]
On the other hand, $D_{\ell_s}$ has constant coefficient one, so the
final summand $t_sx^{h_s}D_{\ell_s}(x)$ contains the nonzero term
$t_sx^{h_s}$.  No earlier summand contains a term of this degree.
Hence $\deg Q\ge h_s$.  Since $s\ge1$, we have
$\deg Q\ge h_s=s(N+1)\ge N+1$.  But
$\deg\Phi_N=\varphi(N)\le N-1$.  Thus $\deg Q\ne\deg\Phi_N$, so the two
irreducible polynomials cannot be associates.  Equation
\eqref{eq:F-and-Q} is therefore the complete factorization of $F$ over
$\mathbb Q$, up to multiplication by nonzero rational constants, and
its two nonassociate irreducible factors are precisely $\Phi_N$ and
$Q$.
\end{proof}

\section{Choosing the primes and prescribing exactly \texorpdfstring{$m$}{m}
coefficients}

It remains to choose $N$.  We need many prime divisors in order to make
$2^r$ large, but Proposition~\ref{prop:assembly} requires the sum of
their individual term counts to remain below $m$.  The prime number
theorem in arithmetic progressions provides precisely this balance.

For $(a,q)=1$, let $\pi(x;q,a)$ denote the number of primes $p\le x$
satisfying $p\equiv a\pmod q$.  Let $\varphi$ be Euler's totient
function and
$\operatorname{li}(x)=\int_2^x dt/\log t$.

\begin{theorem}[Siegel--Walfisz]\label{thm:SW}
Let $A>0$.  There is an absolute constant $c_1>0$ such that, whenever
$1\le q\le(\log x)^A$ and $(a,q)=1$,
\begin{equation}\label{eq:SW}
 \pi(x;q,a)=\frac{\operatorname{li}(x)}{\varphi(q)}
 +O_A\!\left(xe^{-c_1\sqrt{\log x}}\right).
\end{equation}
\end{theorem}

This is \cite[Theorem~12.1, p.~118]{Koukoulopoulos2019}, with the same
notation.

For real $X\ge3$, let
\[
 \mathcal P(X)=
 \{p\le X:p\text{ is prime and }p\equiv2\text{ or }3\pmod5\}.
\]
Write $r(X)=|\mathcal P(X)|$ and
$B(X)=\sum_{p\in\mathcal P(X)}p$.  Taking $q=5$ and $a=2,3$ in
Theorem~\ref{thm:SW} is legitimate for all sufficiently large $X$:
$2$ and $3$ are both coprime to $5$, and the fixed modulus $5$ satisfies
$5\le(\log X)^A$.  Since $\varphi(5)=4$ and
$\operatorname{li}(X)\sim X/\log X$, adding the two resulting formulas
gives
\begin{equation}\label{eq:r-asymptotic}
 r(X)\sim\frac{X}{2\log X}.
\end{equation}

Partial summation converts this counting result into an asymptotic for
the sum of the primes:
\begin{equation}\label{eq:B-partial-summation}
 B(X)=Xr(X)-\int_2^Xr(t)\,dt.
\end{equation}

By \eqref{eq:r-asymptotic}, \(Xr(X)\sim X^{2}/(2\log X)\). Fix \(\eta\in(0,1)\), and choose \(T_{\eta}\) such that \((1-\eta)t/(2\log t)\le r(t)\le(1+\eta)t/(2\log t)\) for every \(t\ge T_{\eta}\). For \(X\ge T_{\eta}\),
\[
\int_{2}^{X}r(t)\,dt=\int_{2}^{T_{\eta}}r(t)\,dt+\int_{T_{\eta}}^{X}r(t)\,dt\]
and \[(1-\eta)\int_{T_{\eta}}^{X}\frac{t}{2\log t}\,dt
\le \int_{T_{\eta}}^{X}r(t)\,dt
\le(1+\eta)\int_{T_{\eta}}^{X}\frac{t}{2\log t}\,dt.
\]
Since the first integral is \(O_{\eta}(1)\), letting \(X\to\infty\) and then \(\eta\to0\) gives \(\int_{2}^{X}r(t)\,dt\sim\int_{2}^{X}t/(2\log t)\,dt\). Moreover,
\[
\left(\frac{X^{2}}{4\log X}\right)'=\frac{X}{2\log X}\left(1-\frac{1}{2\log X}\right)\sim\frac{X}{2\log X},
\]
and thus, \(\int_{2}^{X}t/(2\log t)\,dt\sim X^{2}/(4\log X)\). Substitution into \eqref{eq:B-partial-summation} therefore gives
\begin{equation}\label{eq:B-asymptotic}
 B(X)\sim\frac{X^2}{4\log X}.
\end{equation}

We can now complete the proof of the main theorem.

\begin{proof}[Proof of Theorem~\ref{thm:main-introduction}]
Let $m$ be a sufficiently large integer, and set
\begin{equation}\label{eq:X-and-N}
 X=\frac14\sqrt{m\log m},
 \qquad
 N=\prod_{p\in\mathcal P(X)}p.
\end{equation}
Put $r=r(X)$ and $B=B(X)$.  The integer $N$ is square-free, and each of
its prime divisors is congruent to $2$ or $3$ modulo $5$, as required in
Corollary~\ref{cor:large-small} and
Proposition~\ref{prop:simultaneous-support}.

Since
$\log X=\tfrac12\log m+\tfrac12\log\log m-\log4
       \sim\tfrac12\log m$, equations \eqref{eq:r-asymptotic} and
\eqref{eq:B-asymptotic} give
\begin{equation}\label{eq:r-and-B}
 r\sim\frac14\sqrt{\frac{m}{\log m}},
 \qquad
 B\sim\frac{m}{32}.
\end{equation}
Consequently, once $m$ is sufficiently large,
\begin{equation}\label{eq:usable-r-and-B}
 B<\frac m2,
 \qquad
 r>\frac18\sqrt{\frac{m}{\log m}}.
\end{equation}
After increasing the lower bound on $m$ once more, $r$ lies in the range
required by Proposition~\ref{prop:simultaneous-support}; also $X\ge3$, so
the primes $2$ and $3$ both belong to $\mathcal P(X)$.

We now write $m$ as a sum of prime divisors of $N$, using every prime at
least once.  Choose $\varepsilon\in\{0,1\}$ so that
$m-B-3\varepsilon$ is even, and define
\[
 a=\frac{m-B-3\varepsilon}{2}.
\]
The parity choice makes a an integer, and \(B< \frac{m}{2}\) makes \(a \geq 0\) for sufficiently large \(m\). Begin a finite list $\ell_0,\ldots,\ell_s$ with one copy of
every prime in $\mathcal P(X)$; add one copy of $3$ when $\varepsilon=1$
and add $a$ copies of $2$.  Order the list with $\ell_0=2$.
Every prime divisor of $N$ occurs, and the sum of the entries is exactly
\begin{equation}\label{eq:exact-m}
 \sum_{j=0}^s\ell_j=B+3\varepsilon+2a=m.
\end{equation}

The integer $N$ is square-free and has at least the two distinct prime
divisors $2$ and $3$; every $\ell_j$ divides $N$, and every prime divisor of $N$ occurs in
the list.  Thus every hypothesis of Proposition~\ref{prop:assembly} is
satisfied.  Applying that proposition supplies a
polynomial $F\in\mathbb Q[x]$ whose complete irreducible factorization,
up to rational constants, is $F=\Phi_NQ$.  The first equality in
\eqref{eq:support-counts}, together with \eqref{eq:exact-m}, proves that
$F$ has exactly $m$ nonzero coefficients.

Both irreducible factors have many nonzero coefficients.  Indeed,
Proposition~\ref{prop:simultaneous-support} gives
\[
 |\supp(\Phi_N)|>c_0 2^r.
\]
The second equality in \eqref{eq:support-counts} expresses
$|\supp(Q)|$ as a sum of positive support sizes, one for every entry in
the list.  Each summand is greater than $c_0 2^r$ by the same proposition;
in particular,
\[
 |\supp(Q)|>c_0 2^r.
\]
Thus every irreducible factor of $F$ has more than $c_0 2^r$ nonzero
coefficients.

The lower bound for $r$ in \eqref{eq:usable-r-and-B} implies, after
decreasing the absolute constant $c>0$ to absorb the fixed factor $c_0$,
that
\[
 c_0 2^r>  \exp\!\left(c\sqrt{\frac{m}{\log m}}\right)
\]
for all sufficiently large $m$. Hence every irreducible factor of $F$
satisfies \eqref{eq:main-bound}, which completes the proof.
\end{proof}

\bibliographystyle{plain}
\bibliography{superpolynomial_references}

\end{document}